\documentclass[preprint]{elsarticle}

\makeatletter
\def\ps@pprintTitle{%
	\let\@oddhead\@empty
	\let\@evenhead\@empty
	\def\@oddfoot{\footnotesize\itshape
		{} \hfill}%
	\let\@evenfoot\@oddfoot
}
\makeatother

\usepackage{latexsym}
\usepackage{lineno,hyperref}
\usepackage{indentfirst}
\usepackage{amsxtra}
\usepackage{amssymb}
\usepackage{amsthm}
\usepackage{natbib}
\usepackage{amsmath}
\usepackage{tikz}
\usepackage{amscd}
\newtheorem{teor}{Theorem}
\newtheorem{prop}[teor]{Proposition}

\newtheorem{lemm}[teor]{Lemma}

\theoremstyle{definition}

\newtheorem{es}{Example}
\newtheorem{ess}{Examples}

\begin{document}

\begin{frontmatter}

\title{About a question of Gateva-Ivanova and Cameron on square-free set-theoretic solutions of the Yang-Baxter equation}

\author{Marco Castelli}
\ead{marco.castelli@unisalento.it}

\author{Francesco Catino\corref{cor1}}
\ead{francesco.catino@unisalento.it}

\author{Giuseppina Pinto\corref{cor2}}
\ead{giuseppina.pinto@unisalento.it}

\cortext[cor1]{Corresponding author}

\address{Dipartimento di Matematica e Fisica "Ennio De Giorgi", Universit\`a del Salento, Via Provinciale Lecce-Arnesano, 73100 Lecce, Italy}

%\author[mymainaddress,mysecondaryaddress]{F. Catino}
%\author[mysecondaryaddress]{G. Pinto}


\begin{abstract}
In this paper, we introduce a new sequence $\bar{N}_m$ to find a new estimation of the cardinality $N_m$ of the minimal involutive square-free solution of level $m$. As an application, using the first values of $\bar{N}_m$, we improve the estimations of $N_m$ obtained by Gateva-Ivanova and Cameron in \cite{GiC11} and by Lebed and Vendramin in \cite{lebed2017homology}. Following the approach of the first part, in the last section we construct several new counterexamples to the Gateva-Ivanova's Conjecture.
\end{abstract}

\begin{keyword}
\texttt{Cycle set\sep set-theoretic solution\sep Yang-Baxter equation}
\MSC[2010] 16T25
\end{keyword}

%	\maketitle{}
%	%\date{}
%	{{\center{M. Castelli,\; F. Catino,\; G. Pinto\\}}
%	\footnote{version 17 Luglio 2017}}
%	\pagestyle{plain}
%
%	\begin{abstract}
%	
%	\end{abstract} 
	%\noindent
	%-%Classification MSC 2010: 20F38, 20N02, 20M10, 20F55, 06F05, 16T25\\
	%-%Keywords: Yang-Baxter equation, Set-theoretic solution, Left cycle set.
	%%%%%\maketitle{}

% ----------------------
\end{frontmatter}

%	\makeitle{}
%	%\date{}
%	{{\center{M. Castelli,\; F. Catino,\; G. Pinto\\}}
%	\pagestyle{plain}
%	%\thispagestyle{empty}

% ----------------------
\section{Introduction}
The Yang-Baxter equation is one of the basic equations in mathematical physics. Finding all the solutions of this equation is still an open problem, for this reason Drinfeld \cite{drinfeld1992some} posed the question of finding a particular subclass of these solutions, the so-called  \emph{set-theoretic solution}, i.e. the pair $(X,r)$, where $X$ is a non-empty set and $r:X\times X\to X\times X$  a bijective map, satisfying
$$
r_1r_2r_1 = r_2r_1r_2 ,
$$
where $r_1:= r\times id_X$ and $r_2:= id_X\times r$. 
%Etingof, Schedler and Soloviev (\citeyear{etingof1998set}) and Gateva-Ivanova and Van den Bergh (\citeyear{gateva1998semigroups}) began the study of involutive non-degenerate set-theoretic solutions of the Yang-Baxter equation. 
Recall that, if $\lambda_x:X\to X$ and $\rho_y:X\to X$ are maps such that  
$$
r(x,y) = (\lambda_x(y), \rho_y(x))
$$ 
for all $x,y\in X$, a set-theoretic solution of the Yang-Baxter equation $(X, r)$ is said to be a left [\,right\,] non-degenerate if $\lambda_x\in Sym(X)$ [\,$\rho_x\in Sym(X)$\,] for every $x\in X$ and \textit{non-degenerate} if it is left and right non-degenerate. Moreover a solution is called \emph{involutive} if $r^2=id_{X\times X}$ and \emph{square-free} if $r(x,x)=(x,x)$ for every $x\in X$.
The involutive square-free solutions have received a lot of attention since the work due to Gateva-Ivanova and Van der Bergh \cite{gateva1998semigroups} where they showed that there exist several relations between  square-free set-theoretic solutions of the Yang-Baxter equation and semigroups of I-type, semigroups of skew-polynomial type and Bieberbach groups. In particular, in \cite{gateva2004combinatorial,gateva2008matched} multipermutational square-free solutions of finite order were considered.
Recall that if $(X,r)$ is an involutive non-degenerate solution, it is possible to consider an equivalence relation $\sim$ on $X$ which induces in a classical way an involutive solution  $Ret(X,r):=(X/\sim,\bar{r})$, the so-called \emph{retraction} of $(X,r)$ (for more details see \cite{etingof1998set}). An involutive non-degenerate solution is said to be \textit{retractable} if $|Ret(X)|<|X|$, otherwise it is called \textit{irretractable}. Moreover, an involutive solution is called \emph{multipermutational} of level $m$ if $|Ret^{m-1}(X,r)|>1$ and $|Ret^m(X,r)|=1$, where $Ret^k(X,r)$ is defined inductively as $Ret^k(X,r):=Ret(Ret^{k-1}(X,r))$ for every natural number $k$ greater than $1$.\\
In this context, in 2004 Gateva-Ivanova \cite[Question 2.28]{gateva2004combinatorial} conjectured that every finite square-free involtive non-degenerate solution $(X,r)$ is multipermutational.
Ced\'o, Jespers and Okni\'nski \cite{cedo2010retractability} proved that the conjecture is true if the associated permutation group $\mathcal{G}(X,r)$ is abelian, while some year later Gateva-Ivanova \cite{GiC11} and Ced\'o, Jespers and Okni\'nski \cite{cedo2014braces}, by using different teqniques, showed that if $\mathcal{G}(X,r)$ is abelian, the conjecture is also true without the finiteness of $X$. In 2016, Vendramin \cite{vendramin2016extensions} proved that in full generality the conjecture is false costructing an involutive square-free irretractable solution of cardinality $8$. Some years later, in \cite{bachiller2015family, cacsp2017}, other counterexamples were constructed: even if several examples of irretractable involutive square-free solutions were obtained, the construction of further irretractable square-free solutions still to be very hard. Recently, Vendramin \cite[Problem 19]{vendramin2018skew} formally posed the question of finding other irretractable involutive square-free solutions, emphasizing the research among those having cardinality $9$.\\
On the other hand, because of their links with other algebraic structures, multipermutational square-free involutive solutions have been considered in several papers \cite{GiC11, vendramin2016extensions, cedo2010retractability,cedo2014braces}. In that regard, several methods to construct multipermutational solutions were developed: for example, in \cite{cedo2014braces} Ced\'o, Jespers and Okni\'nski constructed the first family of square-free involutive solutions $X_m$ of level $m$ and abelian associated permutation group. In $2011$, Gateva-Ivanova and Cameron \cite{GiC11} posed the following question:\\
\textbf{Question.}\cite[Open Question 6.13]{GiC11} For each positive integer $m$ denote by $N_m$ the minimal integer so that there exists a square-free involutive multipermutational solution $(X_m,r_m)$ of order $|X_m|=N_m$, and with $mpl(X_m,r_m)=m$. How does $N_m$ depend on $m$?\\
\noindent They showed that $N_m\leq 2^{m-1}+1$ and they noted that for $m\in \{1,2,3\}$ the equality holds. For this reason they conjectured the equality for every $m\in \mathbb{N}$. In 2016, Vendramin \cite[Example 3.2]{vendramin2016extensions} answered in negative sense constructing an involutive square-free solution of cardinality $6$ and multipermutational level $4$. The next year Lebed and Vendramin \cite{lebed2017homology} inspected the involutive finite solutions of small size and they showed that $N_4=6$ and $N_5=8$. Moreover, they considered the relation between two consecutive terms of the succession $N_m$ and they showed that $N_{m+1}\leq 2 N_m$: in this way, since $N_5=8$ they indirectly obtained that $N_m\leq 2^{m-2}$, for every $m>4$.  \\
The goal of this paper is to give a new estimation of $N_m$, introducing a new sequence $\bar{N}_k$, defined as the cardinality of the minimal square-free involutive solution $X$ of multipermutational level $k$, having an automorphism $\alpha$ such that $\sigma_{[k-1]}(x)\neq \sigma_{[k-1]}(\alpha(x))$ for an $x\in X$, where $\sigma_{[n]}$ is the epimorphism from $X$ to $Ret^n(X)$ defined inductively as $\sigma_{[0]}(x):=x$ and $\sigma_{[n]}(x):=\sigma_{\sigma_{[n-1]}(x)}$. In the main result of the paper we will show that 
\begin{equation}\label{disuguaglianza}
N_m\leq \bar{N}_k\cdotp 2^{m-k-1}+1
\end{equation}
for every $k<m$. As an application, working on the first values of the sequence $\bar{N}_k$, we will improve the estimations of $N_m$ due by Gateva-Ivanova and Cameron \cite{GiC11} and Lebed and Vendramin \cite{lebed2017homology}.\\
The main tool of the paper is the algebraic structure of \textit{left cycle sets}, introduced by Rump in \cite{rump2005decomposition} and also considered in several papers (see for example \cite{rump2016quasi,lebed2017homology,vendramin2016extensions,cacsp2017,cacsp2018,cacsp2018quasi,JeP}). Recall that a non-empty set $X$ with a binary operation $\cdot$ is a \emph{left cycle set} if each left multiplication $\sigma_x:X\longrightarrow X,\; y\longmapsto x\cdot y$ is invertible and 
\begin{equation*}\label{cicloide}
(x\cdot y)\cdot (x\cdot z)= (y\cdot x)\cdot (y\cdot z) 
\end{equation*}
for all $x,y,z\in X$.
Moreover, a left cycle set $(X,\cdot)$ is \emph{non-degenerate} if the squaring map $\mathfrak{q}:X\longmapsto X,\; x\mapsto x\cdot x$ is bijective.
Rump proved \cite{rump2005decomposition} that if $(X,\cdot)$ is a non-degenerate left cycle set, the map $r:X\times X\longmapsto X\times X$, defined by $r(x,y)=(\lambda_x(y), \rho_y(x))$, where $\lambda_x(y):=\sigma_x^{-1}(y)$ and $\rho_y(x):=\lambda_x(y)\cdot x$, is a non-degenerate involutive solution of the Yang-Baxter equation. Conversely, if $(X,r)$ is a non-degenerate involutive solution and $\cdotp$ the binary operation given by $x\cdot y:= \lambda_x^{-1}(y)$ for all $x,y\in X$, then $(X,\cdot)$ is a non-degenerate left cycle set. The existence of this bijective correspondence allows to move the study of involutive non-degenerate solutions to non-degenerate left cycle sets. In this context, we prove the inequality (\ref{disuguaglianza}) by a mixture of two well-known extension-tools of left cycle sets: the \textit{one-sided extension} of left cycle sets, developed in terms of set-theoretic solutions \cite{etingof1998set} by Etingof, Schedler and Soloviev, and the \textit{dynamical extension} of left cycle sets developed in \cite{vendramin2016extensions} by Vendramin.\\
In the last section we will see that the same approach is useful to construct further interesting examples of left cycle sets: referring to \cite[Problem 19]{vendramin2018skew}, we provide several counterexamples to the Gateva-Ivanova's Conjecture, in addition to those obtained in \cite{vendramin2016extensions,bachiller2015family,cacsp2017}.

%%%%%%%%%%%%%%%%%%%%%%%%%%%%%%%%%%%%%%%%%%%%%%%%%%%%%%%%%%%%%%%%%%%%%%%%%%%%%%%%%%%%%%
%%%%%%%%%%%%%%%%%%%%%%%%%%%%%%%%%%%%%%%%%%%%%%%%%%%%%%%%%%%%%%%%%%%%%%%%%%%%%%%%%%%%%%%
%%%%%%%%%%%%%%%%%%%%%%%%%%%%%%%%%%%%%%%%%%%%%%%%%%%%%%%%%%%%%%%%%%%%%%%%%%%%%%%%%%%5

\section{Some preliminary results}

A non-empty set $X$ with a binary operation $\cdot$ is a \emph{left cycle set} if the left multiplication $\sigma_x:X\longrightarrow X,\; y\longmapsto x\cdot y$ is invertible and 
\begin{equation}\label{cicloide}
(x\cdot y)\cdot (x\cdot z)= (y\cdot x)\cdot (y\cdot z) 
\end{equation}
for all $x,y,z\in X$. 
A left cycle set $(X,\cdot)$ is \emph{non-degenerate} if the squaring map $\mathfrak{q}:X\longmapsto X,\; x\mapsto x\cdot x$ is bijective.
Rump proved \cite{rump2005decomposition} that if $(X,\cdot)$ is a left cycle set, the map $r:X\times X\to X\times X$ given by $r(x,y)=(\lambda_x(y),\rho_y(x))$, where $\lambda_x(y):=\sigma^{-1}_x(y)$ and $\rho_y(x):=\lambda_x(y)\cdot x$ is a non-degenerate involutive solution of the Yang-Baxter equation. Conversly, if $(X,r)$ is a non-degenerate solution, the binary operation $\cdot$ defined by $x\cdot y:=\lambda^{-1}_x(y)$ for all $x,y\in X$ makes $X$ into a left cycle set. The existence of this correspondence allows to move the study of involutive non degenerate solutions to left cycle sets, as recently done in \cite{vendramin2016extensions, cacsp2017, cacsp2018, lebed2017homology, lebed2016cohomology, rump2016quasi, catino2015construction, cacsp2018quasi}, and clearly to translate in terms of left cycle set the classical concepts related to the non-degenerate involutive set-theoretic solutions.\\
Therefore, a left cycle set is said to be \emph{square-free} if the squaring-map $\mathfrak{q}$ is the identity on $X$. 
The image $\sigma(X)$ of the map $\sigma:X\longrightarrow Sym(X),\; x\mapsto \sigma_x$ can be endowed with an induced binary operation 
$$ 
\sigma_x\cdot\sigma_y:=\sigma_{x\cdot y}
$$
which satisfies (\ref{cicloide}). Rump \cite{rump2005decomposition} showed that $(\sigma(X),\cdot)$ is a non-degenerate left cycle set if and only if $(X,\cdot)$ is non-degenerate. The left cycle set $\sigma(X)$ is called the \emph{retraction} of $(X,\cdot)$.\\
The left cycle set $(X,\cdot)$ is said to be \emph{irretractable} if $(\sigma(X),\cdot)$ is isomorphic to $(X,\cdot)$, otherwise it is called \emph{retractable}.\\
A non-degenerate left cycle set $(X,\cdot)$ is said to be \emph{multipermutational of level $m$}, if $m$ is the minimal non-negative integer such that $\sigma^m(X)$ has cardinality one, where 
$$
\sigma^0(X):= X \quad \textrm{and}\quad \sigma^n(X):=\sigma(\sigma^{n-1}(X)),\quad \textrm{for}\quad n\geq 1 .
$$
In this case we write $mpl(X)=m$.
Obviously, a multipermutational left cycle set is retractable but the converse is not necessarly true.\\
From now on, by a left cycle set we mean a non-degenerate left cycle set.
The permutation group $\mathcal{G}(X)$ of $X$ is the subgroup of $Sym(X)$ generated by the image $\sigma(X)$ of $\sigma$.\\
In order to construct new examples of left cycle sets, Vendramin \cite{vendramin2016extensions} introduced the concept of \emph{dynamical cocycle}.
If $I$ is a left cycle set and $S$ a non-empty set, then   $\alpha:I\times I\times S\longrightarrow Sym(S)$, $(i,j,s)\mapsto \alpha_{i,j}(s,-)$ is a dynamical cocycle if 
\begin{equation}
\alpha_{i\cdot j,i\cdot k}(\alpha_{i,j}(r,s),\alpha_{i,k}(r,t))=\alpha_{j\cdot i,j\cdot k}(\alpha_{j,i}(s,r),\alpha_{j,k}(s,t))
\end{equation}
for all $i,j,k\in I$, $r,s,t\in S$. Moreover, if $\alpha$ is a dynamical cocycle, then the left cycle set $S\times_{\alpha} I:=(S\times I,\cdot)$, where
\begin{equation}
(s,i)\cdot (t,j):=(\alpha_{ij}(s,t),i\cdot j)
\end{equation}
for all $i,j\in I$, $s,t\in S$, is called \emph{dynamical extension} of $I$ by $\alpha$. \\
\noindent A dynamical cocycle $\alpha:I\times I\times S\longrightarrow Sym(S)$ is said to be \textit{constant} \cite{vendramin2016extensions} if $\alpha_{(i,j)}(r,-)=\alpha_{(i,j)}(s,-)$ for all $i,j\in I$, $r,s\in S$. An example of a constant dynamical cocycle, extensively used in \cite{lebed2017homology} because of its simplicity to compute, is the following:

\begin{es}\label{esimp}
Let $S$ be a finite abelian group, $X$ a left cycle set and $f$ a function from $X\times X$ to $S$ such that 
\begin{center}
$f(i,k)+f(i\cdotp j,i\cdotp k)=f(j,k)+f(j\cdotp i,j\cdotp k)$
\end{center}
for all $i,j,k\in X$.
Then, $\alpha:X\times X\times S\longrightarrow Sym(S)$ given by
\begin{center}
$\alpha_{(i,j)}(s,t):=t+f(i,j)$,
\end{center}
for all $i,j\in X$ and $s,t\in S$, is a constant dynamical cocycle.  
\end{es}

\begin{es}\label{esimp2}
Let $X$ be a left cycle set, $k$ a natural number and $S:=\mathbb{Z}/k\mathbb{Z}$. Let $\alpha:X\times X\times S\longrightarrow Sym(S)$ be the function given by
$$
\alpha_{(i,j)}(s,t):=\begin{cases} t & \mbox{if } i=j\\
t+1 & \mbox{if }i\neq j
\end{cases}
$$
Then, $\alpha$ is a constant dynamical cocycle and so $S\times_{\alpha} X$ is a left cycle set.
\end{es}
An important family of dynamical extensions was obtained by Bachiller, Ced\'o, Jespers and Okni\'nski.

\begin{prop}[\cite{bachiller2015family}, Section 2]\label{constrBACH}
Let $A$ and $B$ be non-trivial abelian groups and let $I$ be a set with $|I|>1$. Let $\varphi_1:A\longrightarrow B$ be a function such that $\varphi_1(-a)=\varphi_1(a)$ for every $a\in A$ and let $\varphi_2:B\longrightarrow A$ be a homomorphism. On $X(A,B,I):= A\times B\times I$ we define the following operation  
\begin{center}
$(a,b,i)\cdotp (c,d,j):= \begin{cases} (c,d-\varphi_1(a-c),j), & \mbox{if }i= j \\ (c-\varphi_2(b),d,j), & \mbox{if }\mbox{ }i\neq j\end{cases}$
\end{center}
for all $a,c\in A$, $b,d\in B$ and $i,j\in I$. Then $(X(A,B,I),\cdot)$ is a non-degenerate left cycle set.%and $(X(A,I),\cdot)$ is irretractable square-free whenever $\varphi_1^{-1}(0)=\{0\}$ and $\varphi_2$ is injective.\\
\end{prop}

Recently, we constructed a large family of dynamical extensions that includes the one obtained by Bachiller et al.

\begin{prop}[\cite{cacsp2017}, Theorem 2]\label{costrfinale2}
Let $A,B$ be a non-empty sets, $I$ a non-degenerate left cycle set and $$\beta:A\times A\times I\longrightarrow Sym(B),\quad\gamma:B\longrightarrow Sym(A).$$ 
\vspace{-2mm}
Put $\beta_{(a,c,i)}:=\beta(a,c,i)$, $\gamma_b:=\gamma(b)$ for $a,c\in A$ and $b\in B$. Assume that 
\vspace{1.5mm}
\begin{enumerate}
\item[1)]$\gamma_b\gamma_d=\gamma_d\gamma_b$,

\item[2)] $\beta_{(a,c,i)}=\beta_{(\gamma_b(a),\gamma_b(c),j\cdotp i)}$,

\item[3)]$\gamma_{\beta_{(a,c,i)}(d)}\gamma_{b}=\gamma_{\beta_{(c,a,i)}(b)}\gamma_{d}$,

\item[4)]$\beta_{(a,c,i\cdotp i)}\beta_{(a',c,i)}= \beta_{(a',c,i\cdotp i)}\beta_{(a,c,i)}$
\end{enumerate}

hold for all $a, a', c\in A$, $b,d\in B$ and $i,j\in I$, $i\neq j$.\\
Let $\cdotp$ be the operation on $A\times B\times I$ defined by:
\begin{equation}
\label{opteor}
(a,b,i)\cdotp (c,d,j) := \begin{cases} (c,\; \beta_{(a,c,i)}(d),\;i\cdotp j), & \mbox{if }i=j \\(\gamma_b(c),\;d,\;i\cdotp j), & \mbox{if }i\neq j\end{cases}.
\end{equation}
Then $X(A,B,I,\beta,\gamma):=(A\times B\times I, \cdotp)$ is a non-degenerate left cycle set.
\end{prop}

\section{Left cycle sets and automorphisms}
Before being able to prove our main result in the next section, some preliminary results are requested. At first, one of the needed tool is the concept of automorphism of left cycle sets.
If $X$ is a left cycle set, an element $\alpha\in Sym(X)$ is an \emph{automorphism} of $X$ if $\alpha(x\cdot y)=\alpha(x)\cdot \alpha(y)$ for all $x, y\in X$. \\
In this section, we want to show the importance of the automorphisms group in the study of the left cycle sets. At this purpose we will see that the automorphisms of a left cycle set are useful to construct further examples of left cycle sets and to understand the structure of particular families of left cycle sets. Moreover it is natural ask which is the automorphism group $Aut(X)$ of a given left cycle set $X$. For example, if $X$ is the left cycle set given by $x\cdotp y:=y$ for all $x,y\in X$, then $Aut(X)=Sym(X)$. At this stage of studies, facing this problem in the general case seems to be very hard. However, the following two propositions are useful to find some automorphisms of particular left cycle sets.
\begin{prop}\label{autret}
Let $X$ be a left cycle set and $\alpha\in Aut(X)$. Consider the retraction $\sigma(X)$ and let $\bar{\alpha}:\sigma(X)\longrightarrow \sigma(X)$ be the function given by $\bar{\alpha}(\sigma_x):=\sigma_{\alpha(x)}$ for every $x\in X$. Then $\bar{\alpha}\in Aut(\sigma(X))$.
\end{prop}

\begin{proof}
If $x,y\in X$, then 
\begin{eqnarray}
\sigma_x=\sigma_y &\Leftrightarrow & \forall z\in X\quad  x\cdotp z=y\cdotp z\nonumber \\
&\Leftrightarrow &\forall z\in X \quad \alpha(x)\cdotp \alpha(z)=\alpha(y)\cdotp \alpha(z) \nonumber\\ 
&\Leftrightarrow &\forall z\in X \quad \sigma_{\alpha(x)}=\sigma_{\alpha(y)} \nonumber
\end{eqnarray}
hence $\bar{\alpha}$ is well-defined and injective. With a straightforward calculation, it is possible to show that $\bar{\alpha}$ is an epimorphism. 
\end{proof}

\begin{prop}\label{automor1}
Let $X$ be a left cycle set, $f\in Aut(X)$ and set inductively $X_0:=X$ and $X_m$ the left cycle set on $X_{m-1}\times \mathbb{Z}/2\mathbb{Z}$ given by
$$
(x,s)\cdotp (y,t):=\begin{cases}
(x\cdotp y,t) & \mbox{if }x=y\mbox{ } \\
(x\cdotp y,t+1) & \mbox{if }x\neq y\mbox{ }  
\end{cases}
$$
for all $m\in \mathbb{N}$. Then, the function $f_m:X_m \longrightarrow X_m$ given by $f_m(x,s_1,...,s_m):=(f(x),s_1,...,s_m)$ for every $(x,s_1,...,s_m)\in X_m$ is an automorphism of $X_m$ for every $m\in \mathbb{N}$.
\end{prop}

\begin{proof} We prove the thesis by induction on $m$. If $m=1$ we have
\begin{eqnarray}
f_1((x,s)\cdotp (y,t))&=&f_1(x\cdotp y,t+1-\delta_{x,y}) \nonumber \\
&=&(f(x)\cdotp f(y),t+1-\delta_{f(x),f(y)}) \nonumber \\
&=& (f(x),s)\cdotp (f(y),t) \nonumber \\
&=& f_1(x,s)\cdotp f_1(y,t) \nonumber
\end{eqnarray}
for all $(x,s),(y,t)\in X\times \mathbb{Z}/2\mathbb{Z}=X_1$. Now, let us assume the thesis true for a natural number $m$. Then,
\begin{eqnarray}
& & f_{m+1}((x,s_1,...,s_{m+1})\cdotp (y,t_1,...,t_{m+1})) \nonumber \\
&=&f_{m+1}((x,s_1,...,s_m)\cdotp (y,t_1,...,t_m),t_{m+1}+1-\delta_{(x,s_1,...,s_m),(y,t_1,...,t_m)}) \nonumber \\
&=&f_{m+1}(y,t_1,...,t_m+1-\delta_{(x,s_1,...,s_{m-1}),(y,t_1,...,t_{m-1})},t_{m+1}+1-\delta_{(x,s_1,...,s_m),(y,t_1,...,t_m)}) \nonumber \\
&=&(f(y),t_1,...,t_m+1-\delta_{(x,s_1,...,s_{m-1}),(y,t_1,...,t_{m-1})},t_{m+1}+1-\delta_{(x,s_1,...,s_m),(y,t_1,...,t_m)}) \nonumber \\
%&=&(f(y),t_1,...,t_m+1-\delta_{(f(x),s_1,...,s_{m-1}),(f(y),t_1,...,t_{m-1})},t_{m+1}+1-\delta_{(f(x),s_1,...,s_m),(f(y),t_1,...,t_m)}) \nonumber \\
&=&(f_m((x,s_1,...,s_m)\cdotp (y,t_1,...,t_m)),t_{m+1}+1-\delta_{f_m(x,s_1,...,s_m),f_m(y,t_1,...,t_m)}) \nonumber \\
&=&(f_m(x,s_1,...,s_m)\cdotp f_m(y,t_1,...,t_m),t_{m+1}+1-\delta_{f_m(x,s_1,...,s_m),f_m(y,t_1,...,t_m)}) \nonumber \\
&=&(f_m(x,s_1,...,s_m),s_{m+1})\cdotp (f_m(y,t_1,...,t_m),t_{m+1})	\nonumber \\
&=&(f(x),s_1,...,s_m,s_{m+1})\cdotp (f(y),t_1,...,t_m,t_{m+1})\nonumber \\
&=& f_{m+1}(x,s_1,...,s_{m+1})\cdotp f_{m+1}(y,t_1,...,t_{m+1}) \nonumber 
\end{eqnarray}
hence $f_{m+1}\in Aut(X_{m+1})$.

\end{proof}

Gateva-Ivanova and Cameron \citep{GiC11} and Etingof et al. \cite{etingof1998set} showed that automorphisms of left cycle sets allows to construct other examples of left cycle sets. We prove that, if $X$ is a left cycle set, $\alpha\in Aut(X)$ and $z\notin X$, then, under suitable hypothesis, the retraction $\sigma(X\cup \{z\})$ is isomorphic to the left cycle set having the disjoint union  $\sigma(X)\cup \{\alpha\}$ as underlying set.
\begin{prop}\label{perteo}
Let $X$ be a left cycle set, $\alpha\in Aut(X)$, $z\notin X$ and $(X\cup \{z\},\circ)$ the algebraic structure given by
$$
x\circ y:=\begin{cases}x\cdotp y & \mbox{if }x,y\in X\mbox{ } \\y & \mbox{if }y=z\mbox{ } \\
	 \alpha(y) & \mbox{if }y\in X,x=z.\mbox{ }
	\end{cases}
$$
Then the pair $(X\cup \{z\},\circ)$ is a left cycle set.\\
Moreover, suppose that $\alpha\neq \sigma_x$ for all $x\in X$. Then the retraction $\sigma(X\cup \{z\})$ is isomorphic to the left cycle set $(\sigma(X)\cup \{\sigma_z\},\circ)$ given by
$$
\sigma_x\circ \sigma_y:=\begin{cases}\sigma_x\cdotp \sigma_y & \mbox{if }\sigma_x,\sigma_y\in X\mbox{ } \\\sigma_y & \mbox{if }y=z\mbox{ } \\
	 \sigma_{\alpha(y)} & \mbox{if }y\in X,x=z.\mbox{ }
	\end{cases}
$$
\end{prop}

\begin{proof}
By \cite[Section 2]{etingof1998set}, $(X\cup \{z\},\circ)$ is a left cycle set. Now, since $\alpha\neq \sigma_x$, for every $x\in X$, there is a natural bijection $\phi$ from $\sigma(X\cup \{z\})$ to the disjoint union $\sigma(X)\cup \{\sigma_z\}$ given by $\phi(\sigma_x):=\sigma'_x$, for every $\sigma_x\in \sigma(X\cup \{z\})$, where $\sigma'_x$ and $\sigma_x$ are the left multiplications in $\sigma(X)\cup \{\sigma_z\}$ and $\sigma(X\cup \{z\})$ respectively. It is easy to see that $\phi$ is a homomorphism, so the thesis follows.
\end{proof}

\section{A new estimation of $N_m$}
 
The goal of this section is to provide an estimation of $N_m$ depending on another sequence, which we denote by $\bar{N}_k$, that will allow us to improve the estimation obtained by Lebed and Vendramin \citep{lebed2017homology}.\\
Following Gateva-Ivanova and Cameron \citep{GiC11}, if $X$ is a left cycle set and $n$ a natural number, we indicate by $\sigma_{[n]}$ the epimorphism from $X$ to $\sigma^n(X)$ defined inductively by 
$$\sigma_{[0]}(x):=x\qquad  \sigma_{[n]}(x):=\sigma_{\sigma_{[n-1]}(x)}$$
for all $n\in\mathbb{N}$ and $x\in X$.
The following Lemma, due to Gateva-Ivanova and Cameron, involves the function $\sigma_{[n]}$ and it will be useful in our work.

\begin{lemm}[\cite{GiC11}, Proposition 7.8(3)]\label{gainv} Let $X$ be a finite square-free left cycle set of multipermutational level $k$. Then, the sets $\sigma_{[k-1]}^{-1}(x)$ are $\mathcal{G}(X)$-invariant.
\end{lemm}

We indicate by $\bar{N}_k$ the cardinality of the minimal square-free left cycle set $X$ of level $k$ having an automorphism $\alpha$ such that there exists $x\in X$ with $\sigma_{[k-1]}(x)\neq \sigma_{[k-1]}(\alpha(x))$. For example, let $X:=\{a,b\}$ be the left cycle of level $1$ given by $\sigma_a=\sigma_b:=id_X$ and put $\alpha:=(a\;b)$. Then, $\alpha\in Aut(X)$ and $\sigma_{[0]}(a)\neq \sigma_{[0]}(\alpha(a))$, therefore $\bar{N}_1=2$.

In order to prove the main result of the paper, we need some preliminary results.

\begin{lemm}
Let $X$ be a square-free left cycle set of level $k$, $\{z\}$ a set with a single element such that $z\notin X$ and $\alpha \in Aut(X)$ such that there exists $x\in X$ with $\sigma_{[k-1]}(x)\neq \sigma_{[k-1]}(\alpha(x))$. Then, the left cycle set $(X\cup \{z\},\circ)$ given by
$$
x\circ y:=\begin{cases}x\cdotp y & \mbox{if }x,y\in X\mbox{ } \\y & \mbox{if }y=z\mbox{ } \\
	 \alpha(y) & \mbox{if }y\in X,x=z.\mbox{ }
	\end{cases}
$$
has level $k+1$.
\end{lemm}

\begin{proof}
We prove the thesis by induction on $k$. If $k=1$, then $x\cdotp y=y$ for all $x,y\in X$ and there exist $x,y\in X$ such that $x\neq y$ and $\alpha(x)=y$. This implies that for the left cycle set $(X\cup \{z\},\circ)$ we have that $\sigma_x=id_{X\cup \{z\}}$ for every $x\in X$ and $\sigma_z=\alpha\neq id_{X\cup \{z\}}$, hence $|\sigma(X\cup \{z\})|=2$ and $mpl(X\cup \{z\})=2$.\\
Now, if $X$ has level $k$, then $mpl(X\cup \{z\},\circ)=1+mpl(\sigma(X\cup \{z\}))$ and, by Lemma \ref{gainv}, $\sigma_z\neq \sigma_x$ for every $x\in X$. Hence, by Proposition \ref{perteo}, $\sigma(X\cup \{z\})$ is isomorphic to the left cycle set $(\sigma(X)\cup \{\sigma_z\},\circ)$ given by
$$
\sigma_x\circ \sigma_y:=\begin{cases}\sigma_x\cdotp \sigma_y & \mbox{if }\sigma_x,\sigma_y\in X\mbox{ } \\\sigma_y & \mbox{if }y=z\mbox{ } \\
	 \sigma_{\alpha(y)} & \mbox{if }y\in X,x=z.\mbox{ }
	\end{cases}
$$
Moreover, by Proposition \ref{autret}, we have that $\sigma_{\sigma_z}$ is an element ot $Aut(\sigma(X))$.
If $x$ and $y$ are elements of $X$ such that $\sigma_{[k-1]}(x)\neq \sigma_{[k-1]}(y)$ and $\alpha(x)=y$, it follows that 
$$\sigma_{\sigma_z}(\sigma_x)=\sigma_{z\cdotp x}=\sigma_{\alpha(x)}=\sigma_y$$ 
and 
$$\sigma_{[k-2]}(\sigma_x)=\sigma_{[k-1]}(x)\neq \sigma_{[k-1]}(y)=\sigma_{[k-2]}(\sigma_y)$$
hence we can apply the inductive hypothesis. Therefore, 
$$mpl(X\cup \{z\},\circ)=1+mpl(\sigma(X)\cup \{\sigma_z\})=1+(k-1+1)=k+1$$
and the thesis follows.
\end{proof}

\begin{lemm}\label{stima}
Let $X$ be a square-free left cycle set of level $k$, $\{z\}$ a set with a single element $z\notin X$ and $\alpha \in Aut(X)$ such that there exists $x\in X$ with $\sigma_{[k-1]}(x)\neq \sigma_{[k-1]}(\alpha(x))$. Moreover, set inductively $X_0=Z_0:=X$, $X_m$ the left cycle set on $X_{m-1}\times \mathbb{Z}/2\mathbb{Z}$ given by
$$
(x,s)\cdotp (y,t):=\begin{cases}
(x\cdotp y,t) & \mbox{if }x=y\mbox{ } \\
(x\cdotp y,t+1) & \mbox{if }x\neq y\mbox{ }  
\end{cases}
$$
for all $x,y\in X_{m-1}$ and $s,t\in \mathbb{Z}/2\mathbb{Z}$ and $(Z_m,\circ)$ the algebraic structure given by $Z_m:=X_{m-1}\cup \{z\}$ and
$$
x\circ y:=\begin{cases}x\cdotp y & \mbox{if }x,y\in X_{m-1}\mbox{ } \\y & \mbox{if }y=z\mbox{ } \\
	 (\alpha(y_0),...,y_{m-1}) & \mbox{if }y=(y_0,...,y_{m-1})\in X_{m-1},x=z.\mbox{ }
	\end{cases}
$$
for every $m\in \mathbb{N}$. Then $Z_m$ is a square-free left cycle set of level $k+m$.
\end{lemm}

\begin{proof}
By Propositions \ref{automor1} and \ref{perteo}, the pair $(Z_m,\circ)$ is a square-free left cycle set.
By induction on $m$ we prove that $(Z_m,\circ)$ has level $k+m$. If $m=1$, the thesis follows by the previous proposition. Now, let us suppose the thesis true for a natural number $m$. Since 
$$mpl(Z_{m+1})=1+mpl(\sigma(Z_{m+1}))=1+mpl(\sigma(X_m\cup \{z\}))$$
and, by Propositions \ref{automor1} and \ref{perteo}, $\sigma(X_m\cup \{z\})$ is isomorphic to the left cycle set $(\sigma(X_m)\cup \{\sigma_z\},\circ)$ given by
$$
\sigma_x\circ \sigma_y:=\begin{cases}\sigma_x\cdotp \sigma_y & \mbox{if }\sigma_x,\sigma_y\in X_m\mbox{ } \\\sigma_y & \mbox{if }y=z\mbox{ } \\
	 \sigma_{\alpha(y)} & \mbox{if }y\in X_m,x=z.\mbox{ }
	\end{cases}
$$
it follows that $mpl(Z_{m+1})= 1+mpl(\sigma(X_m)\cup \{\sigma_z\})$. Finally, by \cite[Theorem 10.6 and Corollary 10.7]{lebed2017homology}, $\sigma(X_m)$ is isomorphic to $X_{m-1}$, and so we obtain that $\sigma(X_m)\cup \{\sigma_z\}$ is isomorphic to $Z_m$. By the inductive hypothesis, we have that 
$$mpl(Z_{m+1})=1+mpl(Z_m)=1+k+m, $$
hence the thesis.
\end{proof}

\begin{teor}\label{teorprinci}
Let $N_m$ be the cardinality of the minimal square-free left cycle set of level $m$ and $\bar{N}_k$ the cardinality of the minimal square-free left cycle set $X$ of level $k$ such that there exists an automorphism $\alpha$ and an element $x\in X$ with $\sigma_{[k-1]}(x)\neq \sigma_{[k-1]}(\alpha(x))$. Then, the inequality
\begin{equation}\label{disugprinci}
N_m\leq \bar{N}_k\cdotp 2^{m-k-1}+1
\end{equation}
holds for every $k<m$.
\end{teor}
\begin{proof}
If $r$ is a natural number and $X$ is a left cycle set of level $k$ and cardinality $\bar{N}_k$ having an automorphism $\alpha$ such that there exists $x\in X$ with $\sigma_{[k-1]}(x)\neq \sigma_{[k-1]}(\alpha(x))$, then the left cycle set $Z_r$, constructed as in the previous Lemma, is a square-free left cycle set of level $k+r$ and cardinality $\bar{N}_k\cdotp 2^{r-1}+1$, hence $N_{r+k}\leq  \bar{N}_k\cdotp 2^{r-1}+1$. Setting $m:=r+k$, we obtain
$$N_{m}\leq  \bar{N}_k\cdotp 2^{m-k-1}+1 $$
hence the thesis.
\end{proof}

\section{Some examples and comments}

The goal of this section is to calculate the first values of the sequence $\bar{N}_m$ and to use these to improve the estimations of the sequence $N_m$ obtained by Gateva-Ivanova and Cameron in \cite{GiC11} and by Lebed and Vendramin in \cite{lebed2017homology}. 

\noindent In the following examples we calculate the numbers $\bar{N}_k$ for $k\in \{1,2,3,4\}$.
\begin{ess}
\begin{itemize}
\item[1)] If $X:=\{1,2\}$ is the left cycle set given by $x\cdotp y:=y$ for all $x,y\in X$, then the permutation $\alpha:=(1\;2)$ is an automorphism of $X$. Moreover, $\sigma_{[0]}(1)=1\neq 2=\sigma_{[0]}(2)$, hence $\bar{N}_1=N_1=2$.
\item[2)] The unique square-free left cycle set of size $3$ and level $2$ is given by $\sigma_1=\sigma_2:=id_X$ and $\sigma_3:=(1\;2)$ and the group of automorphism is generated by $\sigma_3$. Since $\sigma_{[1]}(1)=\sigma_{[1]}(2)$, necessarily $\bar{N}_2>3$. Now, let $X:=\{1,2,3,4\}$ be the left cycle set given by 
$$\sigma_1=\sigma_2:=(3\;4)\qquad\sigma_3=\sigma_4:=(1\;2).$$
Then, $X$ has level $2$, $\alpha:=(1\;3)(2\;4)$ is an automorphism of $X$ and $\sigma_{[1]}(1)\neq \sigma_{[1]}(3)$, hence $\bar{N}_2=4$.
\item[3)] We know that $N_3=5$ and, by calculation, there are two left cycle sets of level $3$ and size $5$. The first one is given by $X:=\{1,2,3,4,5\}$, $\sigma_1=\sigma_2:=(3\;4)$, $\sigma_3=\sigma_4:=(1\;2)$ and $\sigma_5:=(1\;3\;2\;4)$. So, the fibers of $\sigma_{[2]}$ are $\sigma_{[2]}^{-1}(1)=\{1,2,3,4\}$ and $\sigma_{[2]}^{-1}(5)=\{5\}$, but an automorphism that maps $5$ to an other element of $X$ can not exists because there are not left multiplications of $X$ conjugate to $\sigma_5$. Using the same argument for the other left cycle set of level $3$ and size $5$, we prove that $\bar{N}_3>5$.
Let $X:=\{1,2,3,4,5,6\}$ be the left cycle set given by 
$$\sigma_1=\sigma_2:=id_X\qquad \sigma_3:=(5\;6)$$
$$\sigma_4:=(1\;2)(5\;6)\qquad\sigma_5:=(3\;4)\qquad \sigma_6:=(1\;2)(3\;4).$$
Then $X$ has level $3$, $\sigma_{[2]}(3)\neq \sigma_{[2]}(5)$ and $\alpha:=(3\;5)(4\;6)\in Aut(X)$, hence $\bar{N}_3= 6$.
\item[4)] By calculation, we know that $\bar{N}_4$ must be greater than $7$. Let $X:=\{1,2,3,4,5,6,7,8\}$ be the left cycle set given by 
$$\sigma_1=\sigma_2:=(7\;8)\qquad \sigma_3:=(5\;6)\qquad \sigma_4:=(1\;7)(2\;8)(5\;6)$$
$$\sigma_5:=(3\;4)\qquad \sigma_6:=(1\;7)(2\;8)(5\;6)\qquad \sigma_7=\sigma_8:=(1\;2).$$
Then $X$ has level $4$, $\sigma_{[3]}(3)\neq \sigma_{[3]}(5)$ and $\alpha:=(3\;5)(4\;6)$ is an automorphism of $X$, hence $\bar{N}_4=8$. 
\end{itemize}
\end{ess}

Unfortunately, for $k>4$ we are not able to calculate the precise value of $\bar{N}_k$. However, the following example allow us to give an estimation of $\bar{N}_5$.

\begin{es}
Let $X:=\{1,2,3,4,5,6,7,8,9,10\}$ be the left cycle set given by 
$$\sigma_1=\sigma_2:=id_X\quad \sigma_3:=(5\;6)\quad \sigma_4:=(1\;2)(5\;6)\quad \sigma_5:=(3\;4)\quad \sigma_6:=(1\;2)(3\;4)$$
$$\sigma_7:=(9\;10)\qquad \sigma_8:=(3\;5)(4\;6)(9\;10)\qquad \sigma_9:=(7\;8)\qquad \sigma_{10}:=(3\;5)(4\;6)(7\;8).$$ 
Then $X$ has level $5$, $\sigma_{[4]}(7)\neq \sigma_{[4]}(9)$ and $\alpha:=(7\;9)(8\;10)\in Aut(X)$, hence $\bar{N}_5\leq 10$.
\end{es}

\noindent By (\ref{disugprinci}), since $\bar{N}_5\leq 10$, it follows that
\begin{equation}\label{nostra}
N_m\leq 2^{m-2}-6\cdotp 2^{m-6}+1
\end{equation}
for every $m>5$. This improve the estimation $N_m\leq 2^{m-1}+1$, for every $m\in \mathbb{N}$, due to Gateva-Ivanova and Cameron \cite{GiC11}, and the estimation 
\begin{equation}\label{vend}
N_m\leq 2^{m-2}
\end{equation}
for every $m>4$, implicitly obtained by Lebed and Vendramin \cite{lebed2017homology}. For example, by \eqref{vend} we have that $N_6\leq 16$ while 
 using \eqref{nostra} we obtain that $N_6\leq 11$. If we want to estimate $N_7$, even if we use the better estimation of $N_6$ obtained in the previous step and even if we use instead of \eqref{vend} the  inequality $N_{m+1}\leq 2 N_m$ we obtain $N_7\leq 22$ while, by \eqref{nostra}, we have that $N_7\leq 21$.
 Using a similar argument one can see that the estimation \eqref{nostra} is actually the best possible for every $m>5$ and  the knowledge of $\bar{N}_k$ for some $k>5$ could be useful to improve our estimation.

\section{Irretractable square-free left cycle sets}

The multipermutational left cycle sets constructed in the previous sections are obtained by a mixture between some dynamical extensions \cite{vendramin2016extensions} and a particular case of an extension-tool developed by Etingof, Schedler and Soloviev called \textit{one-sided extension} (for more details see \cite[Section 2]{etingof1998set}).\\
Using the same approach, in this section we find new examples of irretractable left cycle sets: in particular, according to \cite[Problem 19]{vendramin2018skew}, we obtain several counterexamples to the Gateva-Ivanova's Conjecture that are different from those obtained in \cite{cacsp2017,bachiller2015family}.\\

First of all, we recall the constructions of the irretractable left cycle sets obtained by Bachiller, Ced\'o, Jespers and Okni\'nski  \cite{bachiller2015family} and by the authors \cite{cacsp2017}.

\begin{prop}[\cite{bachiller2015family}, Theorem 3.3(a)-(b)]\label{irretr}
Let $A,B,I,\varphi_1$, $\varphi_2$ and $X(A,B,I)$ be as in Proposition \ref{constrBACH}. 
\begin{itemize}
\item[1)] $X(A,B,I)$ is square-free if ad only if $\varphi_1(0)=0$;
\item[2)] If $\varphi_1^{-1}(0)=\{0\}$ and $\varphi_2$ is injective, then $X(A,B,I)$ is irretractable.
\end{itemize}
\end{prop}

\begin{prop}[\cite{cacsp2017}, Proposition 3]\label{irrpri}
Let $A,B,I,\beta,\gamma$ be as in Proposition \ref{costrfinale2}, $|A|,|B|,|I|>1$. Then if:\newline
$1)$  $A\times A\times \{i\}\cap\beta^{-1}(\beta_{(a,a)}(i,-))\subseteq\{(k,k,i)| k\in A\}$ for all $i\in I$ and $a\in A$,\newline
$2)$ $\gamma$ is injective,\newline
the non-degenerate left cycle set $X(A,B,I,\beta,\gamma)$ is irretractable. 
\end{prop} 

In a similar way for the left cycle sets considered in Section $3$, we can easily obtain some automorphisms of the left cycle sets constructed by Bachiller, Ced\'o, Jespers and Okni\'nski. 

\begin{lemm}\label{automor}
Let $X(A,B,I)$ be the left cycle set of the Proposition \ref{constrBACH}, $m\in Sym(I)$ and let $\psi_m$ the function given by $\psi_m(a,b,i):=(a,b,m(i))$ for every $(a,b,i)\in X(A,B,I)$. Then, $\psi_m$ is an automorphism and the set $G:=\{\psi_m|m\in Sym(I)\}$ is a subgroup of $Aut(X(A,B,I))$ isomorphic to $Sym(I)$.
\end{lemm}

\begin{proof}
Clearly $\psi_m$ is bijective. Moreover, since $i=j$ if and only if $m(i)=m(j)$, we have
\begin{eqnarray}
\psi_m(a,b,i)\cdotp \psi_m(c,d,j) &=& (a,b,m(i))\cdotp (c,d,m(j))\nonumber \\
&=& (c+(1-\delta_{m(i),m(j)})\varphi_2(b),d+\delta_{m(i),m(j)}\varphi_1(a-c),m(j)) \nonumber \\
&=& \psi_m((c+(1-\delta_{i,j})\varphi_2(b),d+\delta_{i,j}\varphi_1(a-c),j)) \nonumber \\
&=& \psi_m((a,b,i)\cdotp (c,d,j)) \nonumber 
\end{eqnarray}
for all $(a,b,i),(c,d,j)\in A\times B\times I$, where $\delta_{k,l}=1$ if $k=l$ and $\delta_{k,l}=0$ otherwise. The rest of the proof is a straightforward calculation.
\end{proof}

From now on, for the left cycle set $X(A,B,I)$, we will indicate by $\psi_m$ the automorphism associated as in the previous Lemma, for every $m\in Sym(I)$.

\begin{prop}\label{uirr}
Let $X(A,B,I)$ be the irretractable left cycle set of the Proposition \ref{irretr}, $(Y,\cdotp)$ a left cycle set and $\alpha:Y\longrightarrow Aut(X(A,B,I))$ an injective function such that $\alpha(Y)\subseteq G$, where $G$ is the subgroup of $Aut(X(A,B,I))$ of the previous Lemma. Moreover, suppose that $\alpha(a\cdotp b)\alpha(a)=\alpha(b\cdotp a)\alpha(b)$ for all $a,b\in Y$ and let $(X(A,B,I)\cup Y,\circ)$ be the algebraic structure given by
$$
x\circ y:=\begin{cases}x\cdotp y & \mbox{if }x,y\in X(A,B,I)\mbox{ } \\y & \mbox{if }y=z\mbox{ } \\
	 \alpha(y) & \mbox{if }y\in X(A,B,I),x=z.\mbox{ }
	\end{cases}
$$
Then, the pair $(X(A,B,I)\cup Y)^{\circ}:=(X(A,B,I)\cup Y,\circ)$ is an irretractable left cycle set.
\end{prop}

\begin{proof}
By \cite[Section 2]{etingof1998set}, we have that $(X(A,B,I)\cup Y)^{\circ}$ is a left cycle set, so it is sufficient to show that $\sigma_x\neq \sigma_y$ for every $x\neq y$. If $x,y\in X(A,B,I)$ or $x,y\in Y$ then clearly $\sigma_x\neq \sigma_y$. Now, suppose that $x:=(a,b,i)\in X(A,B,I)$ and $y\in Y$. If $\alpha(y)\neq id_I$, since $\sigma_x(A\times B\times \{j\})=A\times B\times \{j\}$ and $\sigma_y(A\times B\times \{j\})=A\times B\times \{\alpha(y)(j)\}$ for every $j\in I$, then necessarily $\sigma_x\neq \sigma_y$. Finally, if $\alpha(y)=id_I$ and $\sigma_x=\sigma_y$, then $\sigma_x(c,d,j)=(c,d,j)$ for all $c\in A,d\in B,j\in I$ and this implies $\sigma_{(a,b,i)}=\sigma_{(c,b,i)}$ for every $c\in A$, in contraddiction with the irretractability of $X(A,B,I)$.
\end{proof}

\begin{es}\label{noveuno}
Let $X(A,B,I)$ be the irretractable square-free left cycle set having $8$ element, $m:=id_I$, $Y:=\{m\}$ the left cycle set having $1$ element and $\alpha:Y\longrightarrow Aut(X(A,B,I))$ the function given by $\alpha(m):=\psi_m$. By the previous Proposition, $(X(A,B,I)\cup Y)^{\circ}$ is an irretractable left cycle set; moreover, since $X(A,B,I)$ is square-free, we have that $(X(A,B,I)\cup Y)^{\circ}$ is square-free. Therefore, $(X(A,B,I)\cup Y)^{\circ}$ is a counterexample of the Gateva-Ivanova's Conjecture of cardinality $9$.
\end{es}

\begin{es}\label{novedue}
Let $X(A,B,I)$ be the irretractable square-free left cycle set having $8$ element, $m:=(1\;2)$, $Y:=\{m\}$ the left cycle set having $1$ element and $\alpha:Y\longrightarrow Aut(X(A,B,I))$ the function given by $\alpha(m):=\psi_m$. By Proposition \ref{uirr}, $(X(A,B,I)\cup Y)^{\circ}$ is an irretractable left cycle set; moreover, since $X(A,B,I)$ is square-free, we have that $(X(A,B,I)\cup Y)^{\circ}$ is square-free. Therefore, $(X(A,B,I)\cup Y)^{\circ}$ is a counterexample of the Gateva-Ivanova's Conjecture of cardinality $9$. Since $\sigma_{x}\neq id_{X(A,B,I)\cup Y}$ for every $x\in X(A,B,I)\cup Y$, this left cycle set can not be isomorphic to the previous one.
\end{es}

Finding automorphism of the square-free left cycle set having $8$ element different from those obtained in Lemma \ref{automor} can be useful to find other counterexamples to the Gateva-Ivanova's Conjecture of size $9$, as we can see in the following example.

\begin{es}
Let $X(A,B,I)$ be the irretractable square-free left cycle set having $8$ element and $f:X(A,B,I)\longrightarrow X(A,B,I)$ the function given by $f(a,b,i):=(a+1,b,i)$ for every $(a,b,i)\in A\times B\times I$. Then, $f\in Aut(X(A,B,I))$: indeed $f$ is clearly bijective and
\begin{eqnarray}
f(a,b,i)\cdotp f(c,d,j) &=& (a+1,b,i)\cdotp (c+1,d,j) \nonumber \\
&=& (c+1+(1-\delta_{i,j})\varphi_2(b),d+\delta_{i,j}\varphi_1(a-c),j) \nonumber \\
&=& f(c+(1-\delta_{i,j})\varphi_2(b),d+\delta_{i,j}\varphi_1(a-c),j) \nonumber \\
&=& f((a,b,i)\cdotp(c,d,j)) \nonumber 
\end{eqnarray}
for all $(a,b,i),(c,d,j)\in A\times B\times I$. Therefore, $(X(A,B,I)\cup Y)^{\circ}$ is a counterexample of the Gateva-Ivanova's conjecture of cardinality $9$, where $Y:=\{w\}$ is a left cycle set of size $1$ and $\alpha(w):=f$. Since $\sigma_{x}\neq id_{X(A,B,I)\cup Y}$ for every $x\in X(A,B,I)\cup Y$, this left cycle set can not be isomorphic to the one in Example \ref{noveuno}. Moreover, this left cycle set can not be isomorphic to the previous one: indeed, the previous example has two orbits respect to the associated permutation group, while this left cycle set has three orbits.
\end{es}

However, Proposition \ref{uirr} allow us to obtain many other irretractable left cycle sets arbitrarily large, as we can see in the following examples.

\begin{es}
Let $p$ be an odd prime and $k\in \mathbb{N}$ such that $p^2=4k+1$. Moreover, let $A=B:=\mathbb{Z}/2\mathbb{Z}$, $I:=\{1,...,k\}$, $\varphi_1=\varphi_2=id_A$ and consider the irretractable left cycle set $X(A,B,I)$ as in Proposition \ref{irretr}. Moreover, let $m\in Sym(I)$, $Y:=\{m\}$ the left cycle set of $1$ element and $\alpha:Y\longrightarrow Aut(X(A,B,I))$ given by $\alpha(m):=\psi_m$. Then, $(X(A,B,I)\cup Y)^{\circ}$ is an irretractable square-free left cycle set having $p^2$ elements.
\end{es}

\begin{es}
Let $p$ be a prime number such that $p\equiv 1(mod\;4)$ and $k\in \mathbb{N}$ such that $3p=4k+3$. Moreover, let $A=B:=\mathbb{Z}/2\mathbb{Z}$, $I:=\{1,...,k\}$, $\varphi_1=\varphi_2=id_A$ and consider the irretractable left cycle set $X(A,B,I)$ as in Proposition \ref{irretr}. Moreover, let $m_1:=id_I$, $m_2:=id_I$ and $m_3:=(1\;2)$ and $Y:=\{m_1,m_2,m_3\}$ the left cycle set given by $\sigma_{m_1}=\sigma_{m_2}:=id_Y$ and $\sigma_{m_3}:=(m_1\;m_2)$. Set $\alpha(m_1)=\alpha(m_2):=id_{X(A,B,I)}$ and $\alpha(m_3):=\psi_{m_3}$. Then, $(X(A,B,I)\cup Y,\circ)$, where $\circ$ is the binary operation defined as in Proposition \ref{uirr}, is a square-free left cycle set having $3p$ elements. It is not irretractable since $\sigma_{m_1}=\sigma_{m_2}=id_{X(A,B,I)\cup Y}$.
\end{es}

If we consider the left cycle set $X(A,B,I,\beta,\gamma)$ of Proposition \ref{costrfinale2}, we can easily generalize the argument of Proposition \ref{uirr}.

\begin{lemm}
Let $X(A,B,I,\beta,\gamma)$ be the left cycle set of Theorem \ref{costrfinale2}, $m\in Aut(I)$ such that $\beta_{(a,b)}(i,-)=\beta_{(a,b)}(m(i),-)$ for every $(a,b,i)\in A\times B\times I$ and let $\psi_m$ be the function given by $\psi_m(a,b,i):=(a,b,m(i))$ for every $(a,b,i)\in X(A,B,I)$. Then, $\psi_m$ is an automorphism of $X(A,B,I,\beta,\gamma)$ and the set $G:=\{\psi_m|\psi_m\in Aut(X(A,B,I,\beta,\gamma))\}$ is a subgroup of $Aut(X(A,B,I,\beta,\gamma))$ isomorphic to a subgroup of $Aut(I)$.
\end{lemm}

\begin{prop}\label{uirr2}
Let $X(A,B,I,\beta,\gamma)$ be the irretractable left cycle set of Proposition \ref{irrpri}, $(Y,\cdotp)$ a left cycle set and $\alpha:Y\longrightarrow Aut(X(A,B,I,\beta,\gamma))$ an injective function such that $\alpha(Y)\subseteq G$, where $G$ is the subgroup of $Aut(X(A,B,I,\beta,\gamma))$ of the previous Lemma. Moreover, suppose that $\alpha(a\cdotp b)\alpha(a)=\alpha(b\cdotp a)\alpha(b)$ for all $a,b\in Y$ and let $(X(A,B,I,\beta,\gamma)\cup Y,\circ)$ be the algebraic structure given by
$$
x\circ y:=\begin{cases}x\cdotp y & \mbox{if }x,y\in X(A,B,I,\beta,\gamma)\mbox{ } \\y & \mbox{if }y=z\mbox{ } \\
	 \alpha(y) & \mbox{if }y\in X(A,B,I,\beta,\gamma),x=z.\mbox{ }
	\end{cases}
$$
Then, the pair $(X(A,B,I,\beta,\gamma)\cup Y,\circ)$ is an irretractable left cycle set.
\end{prop}

We leave the proofs of the previous results to the reader because they are similar to the ones of Lemma  \ref{automor} and Proposition \ref{uirr}.\\

\bibliographystyle{elsart-num-sort}
\bibliography{Bibliography}

\end{document}